\documentclass[11pt,a4paper,reqno]{amsart}
\usepackage{fancyhdr}
\usepackage{ifpdf}
\usepackage{amsmath,amsfonts,amsthm,amsopn,amssymb}
\usepackage{verbatim,wasysym,mathrsfs}
\usepackage[left=2.6cm,right=2.6cm,top=3.1cm,bottom=3.1cm]{geometry}
\usepackage{cite,marginnote}
\usepackage{color,enumitem,graphicx}
\usepackage[colorlinks=true,urlcolor=blue, citecolor=red,linkcolor=blue,
linktocpage,pdfpagelabels, bookmarksnumbered,bookmarksopen]{hyperref}
\usepackage[english]{babel}
\usepackage[T1]{fontenc}
\usepackage{lmodern}
\usepackage[hyperpageref]{backref}

\newtheorem{Thm}{Theorem}[section]
\newtheorem{Lem}[Thm]{Lemma}

\newtheorem{Cor}[Thm]{Corollary}
\newtheorem{Prop}[Thm]{Proposition}

\theoremstyle{definition}

\newtheorem{Rem}{Remark}[section]

\newtheorem*{Def*}{Definition}

\newcommand{\R}{{\mathbb R}}

\numberwithin{equation}{section}

\makeindex

\makeatletter
\def\@makefnmark{}
\makeatother

\allowdisplaybreaks

\begin{document}

\title[Nondegeneracy for the Choquard equation in two dimension]{Nondegeneracy of bubble solutions to the Choquard equation in two dimension}

\author[Jinkai Gao]{Jinkai Gao}
\address[Jinkai Gao]{\newline\indent School of Mathematical Sciences,
\newline\indent Nankai University,
\newline\indent Tianjin, 300071, PRC.}
    \email{\href{mailto:jinkaigao@mail.nankai.edu.cn}{jinkaigao@mail.nankai.edu.cn}}

\author[Xinfu Li]{Xinfu Li}
\address[Xinfu Li]{\newline\indent School of Science,
\newline\indent  Tianjin University of Commerce,
\newline\indent Tianjin, 300134, PRC.}
    \email{\href{mailto:lxylxf@tjcu.edu.cn}{lxylxf@tjcu.edu.cn}}

 \author[Shiwang Ma]{Shiwang Ma$^*$}
\address[Shiwang Ma]
{\newline\indent
			School of Mathematical Sciences and LPMC,
			\newline\indent
			Nankai University,
			\newline\indent
			Tianjin 300071, PRC.}
	\email{\href{mailto:shiwangm@nankai.edu.cn}{shiwangm@nankai.edu.cn}}
 \thanks{Corresponding author:
 {\tt Shiwang Ma}.}

\date{\today}

\begin{abstract}
In this paper, we study the following Choquard equation with exponential nonlinearity
\begin{equation*}
     -\Delta u=\left(\int_{\R^{2}}\frac{e^{u(y)}}{|x-y|^{\alpha}}dy\right)e^{u(x)},\quad \text{~in~}\R^{2},
\end{equation*}
where $\alpha\in (0,2)$. Although the classification of solutions to this equation has been established recently, the nondegeneracy of its solutions remains open. Here, we prove the nondegeneracy by combining the integral representation of solutions with the spherical harmonic decomposition. The main result of this paper can be viewed as an extension of the nondegeneracy of solutions for both the planar Liouville equation and the higher-dimensional upper critical Choquard equation.
\end{abstract}

\maketitle


\section{Introduction}
The Choquard equation 
\begin{equation}\label{subcritica equation}
-\Delta u+u=\left(\displaystyle{\int_{\R^{3}}}\frac{u^{2}(y)}{|x-y|}dy\right){u},\ \ \text{in}\ \R^{3},
\end{equation}
arises in various mathematical and physical contexts, including the Polaron Theory \cite{FrohlichHerbert,Pekar}, one-component plasma models \cite{Lieb-1976} and bosonic many-body systems with attractive long-range interactions, particularly in the study of Bose-Einstein condensates \cite{Frohlich-2003,Lewin-2014}. A closely related and mathematically crucial extension of \eqref{subcritica equation} is the following upper critical Choquard equation
\begin{equation}\label{upper critical Choquard equation}
-\Delta u=\left(\displaystyle{\int_{\R^{N}}}\frac{u^{2^{*}_{\alpha}}(y)}{|x-y|^{\alpha}}dy\right){u}^{2^{*}_{\alpha}-1},\ \ \text{in}\ \R^{N},
\end{equation}
where $N\geq 3$, $\alpha\in (0,N)$ and $2^*_{\alpha}:=\frac{2N-\alpha}{N-2}$ is the upper critical exponent in the sense of the Hardy-Littlewood-Sobolev inequality. Note that equation \eqref{upper critical Choquard equation} remains invariant under the rescaling transformation
\begin{equation}
  u(x)\mapsto  u_{\delta,x_{0}}(x):=\delta^{\frac{N-2}{2}}u(\delta (x-x_{0})),\text{~for~}\delta>0 \text{~and~}x_{0}\in\R^{N},
\end{equation}
which preserves both the $\dot{H}^{1}(\R^{N})$ norm and the $ L^{2^*}(\R^{N})$ norm, where $2^*:=\frac{2N}{N-2}$ is the critical Sobolev exponent. This scaling invariance leads to non-compactness in the associated variational problem. 

The fundamental questions concerning the existence and classification of solutions to equation \eqref{upper critical Choquard equation} were completely solved in \cite{Miao-2015,Lei-2018,Du2019,Guo2019,LePhuong-2020}. Specifically, the following classification theorem was established.
\smallskip

\noindent\textbf{Theorem A.}\emph{
Assume that $N\geq 3$ and $\alpha\in (0,N)$. Let $u\in\dot{H}^{1}(\R^{N})$ be a  positive nontrivial weak solution to equation \eqref{upper critical Choquard equation}, then $u$ must be of the form
\begin{equation}
    u(x)=\bar{V}_{\mu,\zeta}(x):=\frac{\left(N(N-2)\right)^{\frac{N-2}{4}}}{(S^{\frac{N-\alpha}{2}}C(N,\alpha))^{\frac{N-2}{2(N+2-\alpha)}}}\left(\frac{\mu}{1+\mu^{2}|x-\zeta|^{2}}\right)^{\frac{N-2}{2}}:=\bar{C}_{N,\alpha}V_{\mu,\zeta}(x),
\end{equation}
where the parameters $\mu>0$ and $ \zeta\in\R^{N}$ represent scaling and translation respectively, $S$ denotes the best Sobolev constant for the embedding $\dot{H}^{1}(\R^{N})\hookrightarrow L^{2^*}(\R^{N})$, given explicitly by
\begin{equation}
     S:=\pi N(N-2)\left(\frac{\Gamma(N/2)}{\Gamma(N)}\right)^{2/N},
\end{equation}
and $C(N,\alpha)$ is the best constant in the Hardy-Littlewood-Sobolev inequality defined by
\begin{equation}
    C(N,\alpha):=\pi^{\frac{\alpha}{2}}\frac{\Gamma\left(\frac{N-\alpha}{2}\right)}{\Gamma\left(N-\frac{\alpha}{2}\right)}\left(\frac{\Gamma(N)}{\Gamma\left(\frac{N}{2}\right)}\right)^{\frac{N-\alpha}{N}},
\end{equation}
where $\Gamma(\cdot)$ is the Gamma function. Moreover, we have
\begin{equation}
    \begin{aligned}
       \int_{\R^{N}} \bar{V}_{\mu,\zeta}^{2^{*}}dx=\frac{\bar{C}_{N,\alpha}^{2^*}S^{\frac{N}{2}}}{(N(N-2))^{\frac{N}{2}}}\text{~and~}  \int_{\R^{N}}\int_{\R^{N}}\frac{\bar{V}_{\mu,\zeta}^{2^{*}_{\alpha}}(y){\bar{V}_{\mu,\zeta}}^{2^{*}_{\alpha}}(x)}{|x-y|^{\alpha}}dydx=\frac{\bar{C}_{N,\alpha}^{2}S^{\frac{N}{2}}}{(N(N-2))^{\frac{N-2}{2}}}.
    \end{aligned}
\end{equation}
In addition
\begin{equation}
   \int_{\R^{N}}\frac{\bar{V}_{\mu,\zeta}^{2^{*}_{\alpha}}(y)}{|x-y|^{\alpha}}dy=\frac{N(N-2)}{\bar{C}_{N,\alpha}^{2^*-2}}\bar{V}_{\mu,\zeta}^{2^*-2^*_{\alpha}}(x).
\end{equation}}

A natural question in studying bubble solutions to \eqref{upper critical Choquard equation} is whether all positive solutions are nondegenerate. This problem had been open for a long time, with only partial results available \cite{Du2019,Giacomoni-Wei-yang-2020,Li-Tang-Xu-2022}, until Li et al. \cite{lixuemei2023nondegeneracy,Li-2025-FM} solved it completely. To be more precise, they obtained

\smallskip

\noindent\textbf{Theorem B.}\label{thm nondegeneracy} \emph{
Assume that $N\geq 3$ and $\alpha\in (0,N)$. Let $\varphi \in \dot{H}^{1}(\R^N)$ be a weak solution to the equation
\begin{equation*}
      -\Delta\varphi=2^{*}_{\alpha} \left(\displaystyle{\int_{\R^N}}\frac{\bar{V}_{\mu,\zeta}^{2^{*}_{\alpha}-1}(y)\varphi(y)}{|x-y|^{\alpha}}dy\right)\bar{V}_{\mu,\zeta}^{2^{*}_\alpha-1}+(2^{*}_\alpha-1)\left(\displaystyle{\int_{\R^N}}\frac{\bar{V}_{\mu,\zeta}^{2^{*}_\alpha}(y)}{|x-y|^{\alpha}}dy\right)\bar{V}_{\mu,\zeta}^{2^{*}_\alpha-2}\varphi,\ \ \textrm{in}\ \R^{N}.
\end{equation*}
Then $\varphi$ admits the representation
\begin{equation}
    \varphi=\sum_{j=1}^{N}a_{j} {\partial_{\zeta_{j}}}\bar{V}_{\mu,\zeta}+b{\partial_{\mu}} \bar{V}_{\mu,\zeta},
\end{equation}
where $a_{j},b\in\R$, $j=1,\cdots,N$.}

In contrast to higher dimensions, the two-dimensional case $N = 2$ exhibits distinct behavior. First, the Sobolev embedding yields $H^1_0(\Omega) \hookrightarrow L^q(\Omega)$ for all $q \geq 1$, but fails to embed into $L^\infty(\Omega)$, when $\Omega$ is a bounded domain. Second, and more strikingly, the modified equation \eqref{upper critical Choquard equation} with $2^*_{\alpha}$ replaced by arbitrary $p\in\R$ admits no positive solutions, as established in \cite{LePhuong-2020}. However, the Trudinger-Moser inequality \cite{Moser-1970,Trudinger-1967} provides the functional framework for analyzing nonlinear elliptic problems with exponential growth. The paradigmatic example is the planar Liouville equation
\begin{equation}\label{Lane-Emden}
    -\Delta u = e^{u},\ \ \text{in}\ \R^{2},
\end{equation}
which arises in a broad range of applications, including astrophysics and combustion theory \cite{Chandrasekhar}, the prescribed Gaussian curvature problem \cite{Kazdan-1975}, the mean field limit of vortices in Euler flows \cite{Caglioti-1995}, and vortices in relativistic Maxwell-Chern-Simons-Higgs theory \cite{Caffarelli-1995}. The mathematical treatment of this problem can be traced back to  at least \cite{FOWLER-1931, Gelcprime}. In particular, the classical solutions with the finite integral condition have been Classified by Chen and Li in their celebrated paper \cite{Chen-Li-1991}.

\smallskip

\noindent\textbf{Theorem C.} \emph{ Let $u$ be a classical solution to the equation \eqref{Lane-Emden} satisfying the condition
\begin{equation}\label{finite integral assume}
    \int_{\R^{2}}e^{u(x)}dx<+\infty.
\end{equation}
Then $u$ must be of the form
\begin{equation}
    u(x)=\bar{U}_{\mu,\zeta}(x):=2\log\left(\frac{2\sqrt{2}\mu}{1+\mu^{2}|x-\zeta|^{2}}\right),
\end{equation}
where $\mu>0$ and $\zeta\in\R^{2}$ are two parameters. Moreover, we have
\begin{equation}
    \int_{\R^{2}}e^{\bar{U}_{\mu,\zeta}(x)}dx=8\pi. 
\end{equation}
}

Later, the nondegeneracy of solutions to \eqref{Lane-Emden} was established by several authors \cite{Esposito-Grossi-Pistoia-2005,Baraket-Pacard-1998,Chen-Lin-2002,ElMehdi-Grossi-2004}.

\smallskip

\noindent\textbf{Theorem D.}\emph{
Let $\varphi \in L^{\infty}(\R^{2})$ be a classical solution to the equation
\begin{equation*}
      -\Delta\varphi=e^{\bar{U}_{\mu,\zeta}}\varphi,\ \ \text{in}\ \R^{2}.
\end{equation*}
Then $\varphi$ admits the representation
\begin{equation}
    \varphi=\sum_{j=1}^{2}a_{j} {\partial_{\zeta_{j}}}\bar{U}_{\mu,\zeta}+b{\partial_{\mu}} \bar{U}_{\mu,\zeta},
\end{equation}
where $a_{j},b\in\R$, $j=1,2$.}

Inspired by the previous work, we are naturally led to the following question: What is the effect of replacing the nonlinearity in \eqref{Lane-Emden} with a nonlocal Choquard-type nonlinearity? In this paper, we will consider this problem. Precisely, we study the following planar Choquard equation with exponential nonlinearity
\begin{equation}\label{slightly subcritical choquard equation}
     -\Delta u=\left(\int_{\R^{2}}\frac{e^{u(y)}}{|x-y|^{\alpha}}dy\right)e^{u(x)},\ \ \text{in}\ \R^{2}.
\end{equation}
First, the classification of solutions to \eqref{slightly subcritical choquard equation} has been extensively studied in recent works \cite{Guo2024JGA,Yang,Niu2025,Gluck2025dcds}. To state their results precisely, we need to introduce the definition of distributional solution. 

\begin{Def*}
Let $\Omega \subset \mathbb{R}^2$ be a (possibly unbounded) domain. For $f \in L^1_{\text{loc}}(\Omega)$, a \emph{distributional solution} to $-\Delta u = f$ in $\Omega$ is any function $u \in L^1_{{loc}}(\Omega)$ satisfying
\begin{equation}\label{distributional formulation}
    -\int_\Omega u\Delta \varphi = \int_\Omega f \varphi,
\end{equation}
for all $\varphi \in C^\infty_c(\Omega)$.
\end{Def*}

The classification result is as follows.
\smallskip

\noindent\textbf{Theorem E.} \emph{Assume that $\alpha\in(0,2)$. Let $u\in L^{1}_{loc}(\R^{2})$ be a distributional solution to the equation \eqref{slightly subcritical choquard equation}
satisfying the condition
\begin{equation}\label{integral condition-1}
    \int_{\R^{2}}e^{\frac{4}{4-\alpha}u(x)}dx<+\infty.
\end{equation}
Then $u\in C^{\infty}(\R^{2})$ and must take the form
\begin{equation}\label{defin-U-xi-lambda}
    u(x)=U_{\mu,\zeta}(x):=\frac{4-\alpha}{2}\log\left( \frac{C_{\alpha}\mu}{1+\mu^{2}|x-\zeta|^{2}}\right),
\end{equation}
where $C_{\alpha}:=\left(\frac{(2-\alpha)(4-\alpha)}{\pi}\right)^{\frac{1}{4-\alpha}}$ is a positive constant, $\mu>0$ and $\zeta\in\R^{2}$ are two parameters. Moreover, we have
\begin{equation}\label{improtant-estimate-2}
    \int_{\R^{2}}e^{\frac{4}{4-\alpha}U_{\mu,\zeta}(x)}dx=\pi C^{2}_{\alpha}\text{~and~}\int_{\R^{2}}\int_{\R^{2}}\frac{e^{U_{\mu,\zeta}(y)}e^{U_{\mu,\zeta}(x)}}{|x-y|^{\alpha}}dydx=2(4-\alpha)\pi.
\end{equation}
}
In addition,
\begin{equation}\label{improtant-estimate-1}
    \int_{\R^{2}}\frac{e^{U_{\mu,\zeta}(y)}}{|x-y|^{\alpha}}dy=\frac{2(4-\alpha)}{C_{\alpha}^{2}}e^{\frac{\alpha}{4-\alpha}U_{\mu,\zeta}(x)}.
\end{equation}

Unlike the higher-dimensional Choquard equation \eqref{upper critical Choquard equation}, the solutions $U_{\mu,\zeta}$ to \eqref{slightly subcritical choquard equation} are sign-changing and fall outside $\dot{H}^{1}(\mathbb{R}^{2})$. The classification of solutions to \eqref{slightly subcritical choquard equation} instead requires the finite integral condition \eqref{integral condition-1}. On the other hand, it is worth noting that equation \eqref{slightly subcritical choquard equation} is invariant under the rescaling transformation
\begin{equation}\label{rescaling}
    u(x)\mapsto u_{\delta,x_{0}}(x):=u(\delta(x-x_{0}))+\frac{4-\alpha}{2}\log \delta,\text{~for~}\delta>0\text{~and~}x_{0}\in\R^{2},
\end{equation}
which also preserves the integrability condition \eqref{integral condition-1}. 

Once the classification of solutions to \eqref{slightly subcritical choquard equation} has been established, it is natural to investigate whether these solutions are nondegenerate. Since the equation \eqref{slightly subcritical choquard equation} is invariant under the rescaling transformation \eqref{rescaling}, we restrict our analysis to the case where \(\mu = 1\) and \(\zeta = 0\) for simplicity. Before stating the main result, we define 
\begin{equation}
    U(x) := U_{1,0}(x) = \frac{4 - \alpha}{2} \log\left( \frac{C_\alpha}{1 + |x|^2} \right),
\end{equation}
and introduce two weighted Hilbert spaces
\begin{equation}
    L^{2}_{w}(\mathbb{R}^2) := \left\{ u: \left\| \frac{1}{1+|x|^{2}} u \right\|_{L^2(\mathbb{R}^2)} < +\infty\right\},
\end{equation}
and
\begin{equation}
    H^{1}_{w}(\mathbb{R}^2) := \left\{ u : \|\nabla u\|_{L^2(\mathbb{R}^2)} + \left\| \frac{1}{1+|x|^{2}} u \right\|_{L^2(\mathbb{R}^2)} < +\infty \right\},
\end{equation}
equipped with the inner products
    \begin{equation}
            \langle u,v\rangle_{ L^{2}_{w}(\mathbb{R}^2)}:=\int_{\R^{2}}\frac{u(x)v(x)}{(1+|x|^{2})^{2}}dx,       
    \end{equation}
and
\begin{equation}
    \langle u,v\rangle_{ H^{1}_{w}(\mathbb{R}^2)}:=\int_{\R^{2}}\nabla u(x)\cdot\nabla v(x)dx+\int_{\R^{2}}\frac{u(x)v(x)}{(1+|x|^{2})^{2}}dx.
\end{equation}
The induced norms are given by
 \begin{equation}
     \begin{aligned}
         \|u\|_{ L^{2}_{w}(\mathbb{R}^2)}:=\left\|\frac{u(x)}{1+|x|^{2}}\right\|_{L^{2}(\R^{2})}\text{~and~}\|u\|_{ H^{1}_{w}(\mathbb{R}^2)}:=\left(\|\nabla u\|_{L^{2}(\R^{2})}^{2}+\left\|\frac{u(x)}{1+|x|^{2}}\right\|_{L^{2}(\R^{2})}^{2}\right
            )^{\frac{1}{2}}.
     \end{aligned}
 \end{equation}   
 
 \begin{Rem}
 Let $\mathbb{S}^2$ denote the unit sphere in $\mathbb{R}^3$ with the standard metric, and $S_{*}$ be the stereographic projection through the south pole (see Section \ref{section-Preliminaries} for more details). Then, the map $u\mapsto u\circ S_{*}$ is an isometry from $L^2_{w}(\mathbb{R}^2)$ to $L^2(\mathbb{S}^2)$, and from $H^1_{w}(\mathbb{R}^2)$ to $H^1(\mathbb{S}^2)$. 
Hence, by the compactness of the embedding $H^1(\mathbb{S}^2) \hookrightarrow L^2(\mathbb{S}^2)$, we deduce the compactness of the embedding $H^1_{w}(\mathbb{R}^2) \hookrightarrow L^2_{w}(\mathbb{R}^2)$.
 \end{Rem}

Our main theorem is as follows.

\begin{Thm}\label{thm-1}
   Assume that $\alpha\in (0,2)$, and $\varphi\in L^{2}_{w}(\R^{2})$ is a distributional solution to the equation 
\begin{equation}\label{eq-2}
      -\Delta\varphi(x)=\left(\displaystyle{\int_{\R^2}}\frac{e^{U(y)}\varphi(y)}{|x-y|^{\alpha}}dy\right)e^{U(x)}+\left(\displaystyle{\int_{\R^2}}\frac{e^{U(y)}}{|x-y|^{\alpha}}dy\right)e^{U(x)}\varphi(x),\ \ \text{in}\ \R^{2}.
\end{equation}
Then $\varphi\in L^{\infty}(\R^{2})\cap C^{\infty}(\R^{2})$ and 
\begin{equation}\label{identity}
    -\int_{\R^{2}}\Delta\varphi(x) dx=\frac{4(4-\alpha)}{C_{\alpha}^{2}}\int_{\R^{2}}e^{\frac{4}{4-\alpha}U(x)}\varphi(x)dx=0.
\end{equation}
Moreover, $\varphi$ admits the representation
\begin{equation}\label{eq-thm-3}
\begin{aligned}
    \varphi&=\sum_{j=1}^{2} a_{j} \partial_{\zeta_{j}}{U}_{\mu,\zeta}|_{\mu=1,\zeta=0}+b \partial_{\mu}{U}_{\mu,\xi}|_{\mu=1,\zeta=0}\\
    &=\sum_{j=1}^{2} a_{j}\frac{(4-\alpha)x_{j}}{1+|x|^{2}}+b\frac{(4-\alpha)}{2}\frac{1-|x|^{2}}{1+|x|^{2}}:=\sum_{j=1}^{2} a_{j}\varphi_{j}+b\varphi_{3},
\end{aligned}
\end{equation}
for some constants $a_{j},b\in\R$,$j=1,2$. In addition, we have $\varphi\in H^{1}_{w}(\R^{2})$ and 
\begin{equation}\label{eq-thm-4}
    \|\varphi\|_{H^{1}_{w}(\R^{2})}\leq A_{\alpha}\|\varphi\|_{L^{2}_{w}(\R^{2})},
\end{equation}
 where $A_{\alpha}:=\sqrt{4(4-\alpha)+1}$. 
\end{Thm}

Since $\dot{H}^{1}(\mathbb{R}^{2})\subset L^{2}_{w}(\R^{2})$ (see Appendix \ref{appendix}), we then obtain the following corollary.
\begin{Cor}\label{thm-weak-solution}
Assume that $\alpha\in (0,2)$ and $\varphi\in \dot{H}^{1}(\mathbb{R}^{2})$ is a weak solution to the equation \eqref{eq-2}. Then, the conclusions of Theorem \ref{thm-1} hold.
\end{Cor}

\begin{Rem}
\begin{enumerate}
    \item From \eqref{slightly subcritical choquard equation} and \eqref{defin-U-xi-lambda}, we immediately deduce that each $\varphi_{j}\in L^{\infty}(\R^{2})\cap C^{\infty}(\R^{2})$ satisfies the linear equation \eqref{eq-2}, where $j=1,2,3$. Theorem \ref{thm-1} further establishes that these $\varphi_{j}$ are all possible solutions to equation \eqref{eq-2}. In particular, the integral identity \eqref{identity}, whose naturalness follows from \eqref{improtant-estimate-2} and \eqref{improtant-estimate-1}, plays a key role in the proof of Theorem \ref{thm-1}, see Proposition \ref{prop-4.5} for more details.

\item As $\alpha\to0$, equation \eqref{eq-2} is formally reduced to 
\begin{equation}\label{eq-alpha-0}
   -\Delta\varphi=\frac{(8/\pi)^{1/2}C_{0,\varphi}}{(1+|x|^{2})^{2}}+\frac{8}{(1+|x|^{2})^{2}}\varphi(x),\ \ \text{in}\ \R^{2},
\end{equation}
where $C_{0,\varphi}:=(8/\pi)^{1/2}\displaystyle{\int_{\R^2}}\frac{1}{(1+|y|^{2})^{2}}\varphi(y)dy$. Given that $\varphi\in L^{2}_{w}(\R^{2})$, the H\"older inequality yields
\begin{equation}
\begin{aligned}
     C_{0,\varphi}\leq (8/\pi)^{1/2}\left(\int_{\R^2} \frac{1}{(1+|y|^{2})^{2}}dy\right)^{\frac{1}{2}}\left(\int_{\R^2} \frac{\varphi^{2}(y)}{(1+|y|^{2})^{2}}dy\right)^{\frac{1}{2}}<+\infty,
\end{aligned}
\end{equation}
which implies that $C_{0,\varphi}$ is well-defined. Moreover, following the arguments analogous to those in Proposition \ref{prop-distributional}--Proposition \ref{prop-integral representation} below, we conclude that $\varphi\in L^{\infty}(\R^{2})\cap C^{\infty}(\R^{2})$ and $C_{0,\varphi}=0$. The equation \eqref{eq-alpha-0} therefore simplifies to
\begin{equation}
     -\Delta\varphi=\frac{8}{(1+|x|^{2})^{2}}\varphi(x),\ \ \text{in}\ \R^{2},
\end{equation}
which together with Theorem D gives that
\begin{equation}
    \varphi=\sum_{j=1}^{2} a_{j}\frac{4x_{j}}{1+|x|^{2}}+b\frac{2(1-|x|^{2})}{1+|x|^{2}},
\end{equation}
for some constants $a_{j},b\in\R$, $j=1,2$. These results match exactly what we obtain by taking the limit $\alpha\to0$ in the conclusion \eqref{eq-2} of Theorem \ref{thm-1}.

\end{enumerate}
\end{Rem}

\noindent\textbf{Methods and difficuliculties.}
Due to the presence of convolution terms, the ordinary differential equations methods used for the local Liouville equation \eqref{Lane-Emden} are no longer applicable. Instead, we employ the method introduced in \cite{lixuemei2023nondegeneracy}, which relies primarily on the spherical harmonic decomposition and the Funk-Hecke formula. However, the situation differs significantly, and the proof requires careful treatment. First, we do not require the solutions to be bounded or classical; instead, we only assume that they belong to $L^{2}_{w}(\mathbb{R}^{2})$ and satisfy equation \eqref{eq-2} in the distributional sense. Indeed, any function $\varphi$ satisfying the growth condition
\[
|\varphi(x)| \leq C(1+|x|^{\tau}) \quad \text{for some} \quad \tau \in [0,1),
\]
automatically belongs to $L^{2}_{w}(\mathbb{R}^{2})$. This relaxed regularity framework requires careful argument. Second, the fundamental solution for the Laplace operator in two dimensions is sign-changing and exhibits markedly different behavior compared to higher-dimensional cases. Moreover, unlike in higher dimensions, solutions to equation \eqref{eq-2} may not decay to zero at infinity. Consequently, the corresponding integral representation formula may contain a non-zero constant term, so several additional technical estimates are required in the proof.



The paper is organized as follows. We review several preliminary results in Section \ref{section-Preliminaries}.  In Section \ref{section-proof-thm-1}, we first establish the regularity properties and derive an integral representation formula for solutions to \eqref{eq-2}. Then by using the stereographic projection and the spherical harmonic decomposition, we prove Theorem \ref{thm-1}. The Appendix \ref{appendix} contains the proof of $\dot{H}^{1}(\mathbb{R}^{2}) \subset L^{2}_{w}(\mathbb{R}^{2})$.

\smallskip

\noindent\textbf{Notations.}
Throughout this paper, we adopt the following notations.
\begin{enumerate} 
     \item  Let $f,g: X\to \R^{+}\cup\{0\}$ be two nonnegative function defined on some set $X$. We write $f\lesssim g$ or $g\gtrsim f$, if there exists a constant $C>0$ independent of $x$ such that $f(x)\leq C g(x)$ for any $x\in X$, and $f\sim g$  means that $f\lesssim g$ and $g\lesssim f$.
     \item The homogeneous Sobolev space $\dot{H}^{1}(\R^{N})$ is defined as
     \begin{equation}
         \dot{H}^{1}(\R^{N}):=\{u\in L^{1}_{loc}(\R^{N}):\nabla u\in L^{2}(\R^{N})\}.
     \end{equation}
     The space $\text{BMO}(\R^{N})$ of bounded mean oscillations is the set of locally integrable functions $f$ such that
     \begin{equation}
         \|f\|_{\text{BMO}(\R^{N})}:=\sup_{Q}\frac{1}{|Q|}\int_{Q}|f(x)-f_{Q}|dx<+\infty\text{~~with~~}f_{Q}:=\frac{1}{Q}\int_{Q}f(x)dx.
     \end{equation}
     The above supremum is taken over the set of Euclidean cubes $Q$.
     \item We define the Japanese bracket $\langle x \rangle := \left(1 + |x|^2\right)^{1/2}$. The unit sphere in $\mathbb{R}^N$ is denoted by $\mathbb{S}^{N-1}$, i.e.,
\begin{equation}
    \mathbb{S}^{N-1}:=\left\{\xi=(\xi_{1},\cdots,\xi_{N})\in\mathbb{R}^{N} \quad\left|\quad\sum_{j=1}^{N}\xi_{j}^{2}=1\right.\right\}.
\end{equation}
\item For any $1\leqslant p<\infty$, we denote by $L^{p}\left(\mathbb{R}^{N}\right)$ and $L^{p}\left(\mathbb{S}^{N}\right)$ the spaces of real-valued $p$-th power integrable functions on $\mathbb{R}^{N}$ and $\mathbb{S}^{N}$. Moreover, we equip $L^{p}\left(\mathbb{R}^{N}\right)$ and $L^{p}\left(\mathbb{S}^{N}\right)$ with the norms:
\begin{equation}
    \left\|f\right\|_{L^{p}(\R^{N})}:=\left(\int_{\mathbb{R}^{N}}|f\left(x\right)|^{p}\,\mathrm{d}x\right)^{\frac{1}{p}},\text{~for~}f\in L^{p}\left(\mathbb{R}^{N}\right)
\end{equation}
and
\begin{equation}
\left\|F\right\|_{L^{p}(\mathbb{S}^{N})}:=\left(\int_{\mathbb{S}^{N}}|F\left(\xi\right)|^{p}\,\mathrm{d}\xi\right)^{\frac{1}{p}},\text{~for~}F\in L^{p}\left(\mathbb{S}^{N}\right),
\end{equation}
where $\mathrm{d}\xi$ is the standard volume element on the sphere $\mathbb{S}^{N}$.
\end{enumerate}

\section{Preliminaries}\label{section-Preliminaries}

First, we recall the Hardy-Littlewood-Sobolev (HLS) inequality \cite{Lieb2001}, a key tool for estimating the nonlocal term.
\begin{Lem} \label{lema 2.1} 
Suppose $N\geq1$, $\alpha\in(0,N)$ and $\theta,\,r>1$ with $\frac{1}{\theta}+\frac{1}{r}+\frac{\alpha}{N}=2$. Let $f\in L^{\theta}(\R^N)$ and $g\in L^{r}(\R^N)$, then there exists a sharp constant $C(\theta,r,\alpha,N)$, independent of $f$ and $g$, such that
\begin{align}\label{HLS}
\displaystyle{\int_{\R^N}}\displaystyle{\int_{\R^N}}\frac{f(x)g(y)}{|x-y|^{\alpha}}dxdy\leq C(\theta,r,\alpha,N)\|f\|_{L^{\theta}(\R^N)}\|g\|_{L^{r}(\R^N)}.
\end{align}
If $\theta=r=\frac{2N}{2N-\alpha}$, then
\begin{equation}\label{definition of C N alpha}
    C(\theta,r,\alpha,N)=C(N,\alpha):=\pi^{\frac{\alpha}{2}}\frac{\Gamma\left(\frac{N-\alpha}{2}\right)}{\Gamma\left(N-\frac{\alpha}{2}\right)}\left(\frac{\Gamma(N)}{\Gamma\left(\frac{N}{2}\right)}\right)^{\frac{N-\alpha}{N}}.
\end{equation}
In this case, the equality in \eqref{HLS} holds if and only if $f\equiv (const.)\, g$, where
$$g(x)=A\left(\frac{1}{\gamma^{2}+|x-a|^{2}}\right)^{\frac{2N-\alpha}{2}},\quad \text{for some $A\in \mathbb{C}$, $0\neq\gamma\in\R$ and $a\in\R^N$.}$$
\end{Lem} 



Next, we present a useful integral estimate established in \cite{lixuemei2023nondegeneracy}.
\begin{Lem}\label{lemma:integral_estimate}
    Let $N\geq1$ $\lambda \in (0, N)$ and $\theta + \lambda > N$. Then
   \begin{equation}
       \int_{\R^N} \frac{1}{|x-y|^\lambda} \frac{1}{\langle{y}\rangle^{\theta}} \, \mathrm{d}y 
    \lesssim 
    \begin{cases}
        \langle{x}\rangle^{N-\lambda-\theta}, & \text{if } \theta < N, \\
         \langle{x}\rangle^{-\lambda} \left(1 + \log  \langle{x}\rangle\right), & \text{if } \theta = N, \\
        \ \langle{x}\rangle^{-\lambda}, & \text{if } \theta > N.
    \end{cases}
   \end{equation}
\end{Lem}

To state the integral representation formula for solutions to \eqref{eq-2}, we introduce the operator $K$ on $L^{1}(\R^{2})$ given by
\begin{equation}\label{defin-K}
    K(f)(x):=\frac{1}{2\pi}\int_{\R^{2}}\log\left(\frac{1+|y|}{|x-y|}\right)f(y)dy.
\end{equation}
Furthermore, we recall the following properties of the operator $K$ established in \cite{Gluck2025dcds}.

\begin{Lem}\label{lem-K-2}
If \(f\in L^{1}(\mathbb{R}^{2})\) then \(Kf\), as defined in \eqref{defin-K}, satisfies \(-\Delta(Kf)=f\) in the sense of distributions on \(\mathbb{R}^{2}\).
\end{Lem}

\begin{Lem}\label{lem-K-1}
Let \(K\) be the operator defined in \eqref{defin-K}. Then 
\begin{enumerate}[label=(\arabic*),font=\upshape]
    \item If \(f\in L^{1}(\mathbb{R}^{2})\) then \(Kf\in W_{{loc}}^{1,1}(\mathbb{R}^{2})\) and for every \(i\in\{1,2\}\) the following equality holds in the sense of \(L^{1}_{{loc}}(\mathbb{R}^{2})\):
    \[
    \partial_{i}Kf(x)=-\frac{1}{2\pi}\int_{\mathbb{R}^{2}}\frac{x_{i}-y_{i}}{|x-y|^{2}}f(y)\;\mathrm{d}y.
    \]

    \item If \(f\in L^{1}(\mathbb{R}^{2})\cap L^{p}_{{loc}}(\mathbb{R}^{2})\) for some \(p>2\), then \(Kf\in W_{{loc}}^{1,\infty}(\mathbb{R}^{2})\).

    \item If \(f\in L^{1}(\mathbb{R}^{2})\cap L^{\infty}(\mathbb{R}^{2})\) then \(Kf\in C^{1}(\mathbb{R}^{2})\).
\end{enumerate}  
\end{Lem}

Finally, we define the stereographic projection $\mathcal{S}:\R^{2}\to\mathbb{S}^{2}\setminus\{(0,0,-1)\}$ by
\begin{equation}
    \mathcal{S}(x):=\left(\frac{2x}{1+|x|^{2}},\frac{1-|x|^{2}}{1+|x|^{2}}\right),
\end{equation}
with its inverse $\mathcal{S}^{-1}:\mathbb{S}^{2}\setminus\{(0,0,-1)\}\to\R^{2}$ given by
\begin{equation}
    \begin{aligned}\label{defin-S--1}
        \mathcal{S}^{-1}(\xi_{1},\xi_{2},\xi_{3}):=\left(\frac{\xi_{1}}{1+\xi_{3}},\frac{\xi_{2}}{1+\xi_{3}}\right).
    \end{aligned}
\end{equation}
Setting $\rho(x):=\left(\frac{2}{1+|x|^{2}}\right)^{\frac{1}{2}}$, we recall from\cite{Frank-2012,Lieb2001} that
\begin{equation}\label{eq-S-x-S-y-x-y}
   g_{ij}=\rho^{4}(x)\delta_{ij},\quad|\mathcal{S}x-\mathcal{S}y|=|x-y|\rho(x)\rho(y)\text{~and~}d\xi=\rho^{4}(x)dx,
\end{equation}
where $g_{ij}\left(1\leqslant i,\,j\leqslant 3\right)$ stands for the metric on $\mathbb{S}^{2}$, which is inherited from $\mathbb{R}^{3}$. 
Therefore, for any $F\in L^{1}(\mathbb{S}^{2})$, we have the following identity
\begin{equation}\label{identity-integral}
    \int_{\mathbb{S}^{2}}F(\xi)d\xi=\int_{\R^{2}}F(\mathcal{S}x)\rho^{4}(x)dx.
\end{equation}

To establish connections between functions on $\mathbb{R}^2$ and $\mathbb{S}^2$, we compose the functions in $\R^{2}$ and $\mathbb{S}^{2}$ with the maps $\mathcal{S}_{*}$ and $\mathcal{S}^{*}$. For any $f:\R^{2}\to\R$, we denote the weighted pushforward map $\mathcal{S}_{*}f:\mathbb{S}^{2}\setminus\{0,0,-1\}\to \R$ by 
\begin{equation}\label{defin-S-*}
    \mathcal{S}_{*}f(\xi)=f(\mathcal{S}^{-1}\xi),
\end{equation}
and for any $F:\mathbb{S}^{2}\setminus\{0,0,-1\}\to\R$, we denote the weighted pullback map $\mathcal{S}^{*}F:\R^{2}\to\R$ by 
\begin{equation}
    \mathcal{S}^{*}F(x):=F(\mathcal{S}x).
\end{equation}

We now introduce spherical harmonic functions, which are eigenfunctions of the Laplace-Beltrami operator on $\mathbb{S}^2$ (see \cite{Atkinson2012, Dai-Xu-2013}). These admit the following orthogonal decomposition:
\begin{equation}\label{orthogonal decomposition}
L^2(\mathbb{S}^2) = \bigoplus_{k=0}^{\infty} \mathcal{H}_k, 
\end{equation}
where $\mathcal{H}_k $ ($ k \geq 0$) denote the mutually orthogonal subspaces of the restriction on $ \mathbb{S}^2$ of real, homogeneous harmonic polynomials of degree $k$, and
\begin{equation}
    \dim \mathcal{H}_k= 2k+1.
\end{equation}
We will use $\{ Y_{k,j} \mid k\geq0\text{~and~}1 \leqslant j \leqslant \dim \mathcal{H}_k \}$ to denote an orthonormal basis of $\mathcal{H}_k $. In particular, the first-order harmonics are given by
\begin{equation}
    Y_{1,j}(\xi) = \sqrt{\dfrac{3}{2\pi}} \xi_j, \quad 1 \leqslant j \leqslant 3,
\end{equation}
and
\begin{equation}\label{defin-H-1}
   \mathcal{H}_1 = \text{span}\{ \xi_j \mid 1 \leqslant j \leqslant 3 \}.
\end{equation}

Furthermore, the Funk-Hecke formula \cite{Atkinson2012, Dai-Xu-2013} yields the following estimates for spherical harmonics.

\begin{Lem}\label{lem-2.4}
Let $\alpha \in (0, 2)$, the integer $k \geqslant 0 $ and $Y \in \mathcal{H}_k$. Then we have
\begin{equation}
\int_{\mathbb{S}^2} \dfrac{1}{|\xi - \eta|^\alpha} Y(\eta) d\eta = \mu_k(\alpha) Y(\xi)\text{~and~}
        \int_{\mathbb{S}^{2}}\log(|\xi-\eta|)Y(\eta)d\eta=\tilde{\mu}_{k} Y(\xi),
    \end{equation}
where
\begin{equation}\label{defin-mu-k-alpha}
  \mu_k(\alpha) = 2^{2-\alpha} \pi \dfrac{\Gamma\left(k+ \frac{\alpha}{2} \right) \Gamma\left(\frac{2-\alpha}{2} \right)}{\Gamma\left( \frac{\alpha}{2} \right) \Gamma\left( k +2- \frac{ \alpha}{2} \right)}\text{~and~}
     \tilde{\mu}_{k}:=\begin{cases}
            2\pi(\log2-1),&\text{~for~}k=0,\\
            -\frac{2\pi}{k(k+1)},&\text{~for~}k\geq 1.\\
        \end{cases}
\end{equation}
In particular,
\begin{equation}
\mu_0(\alpha)=\frac{2^{3-\alpha}\pi}{2-\alpha}\text{~and~}\mu_1(\alpha) = \frac{2^{3-\alpha}\pi\alpha}{(2-\alpha)(4-\alpha)}.    
\end{equation}
\end{Lem} 
\begin{Rem}
Simple calculations give that
\begin{equation}
\mu_k(\alpha) > \mu_{k+1}(\alpha),  \text{~for all~} k \geqslant 0.    
\end{equation} 
\end{Rem}

\section{Proof of Theorem \ref{thm-1}}\label{section-proof-thm-1}


In the following, we define the operator \(\mathfrak{N}(\varphi)\) by
\begin{equation}\label{eq-defin-N-varphi}
    \begin{aligned}
        \mathfrak{N}(\varphi)(x)&:=\left(\displaystyle{\int_{\R^2}}\frac{e^{U(y)}\varphi(y)}{|x-y|^{\alpha}}dy\right)e^{U(x)}+\left(\displaystyle{\int_{\R^2}}\frac{e^{U(y)}}{|x-y|^{\alpha}}dy\right)e^{U(x)}\varphi(x)\\
        &:=\mathfrak{N}_{1}(\varphi)(x)+\mathfrak{N}_{2}(\varphi)(x)
    \end{aligned}
\end{equation}
and begin by proving that the solutions to \eqref{eq-2} are indeed bounded and smooth.

\begin{Prop}\label{prop-distributional}
    Assume that $\alpha\in (0,2)$, and $\varphi\in L^{2}_{w}(\mathbb{R}^2)$ is a distributional solution to the equation \eqref{eq-2}. Then $\varphi \in L^{\infty}(\R^{2})$.
\end{Prop}

\begin{proof}
For any $\phi(x)\in C^{\infty}_{c}(\R^{2})$, by the HLS inequality and the H\"older inequality, we have
\begin{equation}
    \begin{aligned}
        \int_{\R^{2}}\mathfrak{N}_{1}(\varphi)(x)\phi(x)dx&=\int_{\R^{2}}\int_{\R^{2}}\frac{e^{U(y)\varphi(y)e^{U(x)}\phi(x)}}{|x-y|^{\alpha}}dydx\\
        &\lesssim \left(\int_{\R^{2}}e^{\frac{4}{4-\alpha}U(y)}|\varphi|^{\frac{4}{4-\alpha}}(y)dy\right)^{\frac{4-\alpha}{4}} \left(\int_{\R^{2}}e^{\frac{4}{4-\alpha}U(x)}|\phi|^{\frac{4}{4-\alpha}}(x)dy\right)^{\frac{4-\alpha}{4}}\\
        &\lesssim\left( \int_{\R^{2}}\frac{|\varphi|^{\frac{4}{4-\alpha}}(y)}{(1+|y|^{2})^{2}}dy\right)^{\frac{4-\alpha}{4}}\left(\int_{\R^{2}}\frac{1}{(1+|x|^{2})^{2}}dx\right)^{\frac{4-\alpha}{4}}\\
        &\lesssim \left( \int_{\R^{2}}\frac{1}{(1+|y|^{2})^{2}}dy\right)^{\frac{2-\alpha}{4}}\left( \int_{\R^{2}}\frac{\varphi^{2}(y)}{(1+|y|^{2})^{2}}dy\right)^{\frac{1}{2}}<+\infty.
    \end{aligned}
\end{equation}
Notice that, by \eqref{improtant-estimate-1}
\begin{equation}\label{proof-distributional-solution-2}
    \begin{aligned}
        \left(\displaystyle{\int_{\R^2}}\frac{e^{U(y)}}{|x-y|^{\alpha}}dy\right)e^{U(x)}\varphi(x)=\frac{2(4-\alpha)}{C_{\alpha}^{2}}e^{\frac{4}{4-\alpha}U(x)}\varphi(x)=\frac{2(4-\alpha)}{(1+|x|^{2})^{2}}\varphi(x).
    \end{aligned}
\end{equation}
Then we obtain
\begin{equation}
    \begin{aligned}
        \int_{\R^{2}}\mathfrak{N}_{2}(\varphi)(x)\phi(x)dx&= \int_{\R^{2}}\frac{2(4-\alpha)}{(1+|x|^{2})^{2}}\varphi(x)\phi(x)dx\\
        &\lesssim\left(\int_{\R^{2}}\frac{\varphi^{2}(x)}{(1+|x|^{2})^{2}}dx\right)^{\frac{1}{2}}\left(\int_{\R^{2}}\frac{\phi^{2}(x)}{(1+|x|^{2})^{2}}dx\right)^{\frac{1}{2}}<+\infty.
    \end{aligned}
\end{equation}
Hence $\mathfrak{N}(\varphi)(x)\in L^{1}_{loc}(\R^{2})$ and the distributional formulation \eqref{distributional formulation} is well-defined.

Fix $q\in(1,2)$ such that $\alpha q\in(0,2)$. We now prove that $\mathfrak{N}(\varphi)(x)\in L^{q}(\R^{2})$. Applying the Minkowski inequality, Lemma \ref{lemma:integral_estimate} and the H\"older inequality, we obtain that
\begin{equation}
    \begin{aligned}
         \left(\int_{\R^{2}}\mathfrak{N}_{1}(|\varphi|)^{q}(x)dx\right)^{\frac{1}{q}}&\leq \int_{\R^2}\left(\int_{\R^{2}}\frac{e^{qU(x)}}{|x-y|^{q\alpha}}dx\right)^{\frac{1}{q}}e^{U(y)}|\varphi(y)|dy\lesssim  \int_{\R^2}\frac{|\varphi(y)|}{(1+|y|)^{4}}dy\\
         &\lesssim \left(\int_{\R^2}\frac{1}{(1+|y|)^{4}}dy\right)^{\frac{1}{2}}\left(\int_{\R^2}\frac{\varphi^{2}(y)}{(1+|y|^{2})^{2}}dy\right)^{\frac{1}{2}}<+\infty.
    \end{aligned}
\end{equation}
On the other hand, by \eqref{proof-distributional-solution-2} and the H\"older inequality, we have
\begin{equation}
    \begin{aligned}
        \left(\int_{\R^{2}}\mathfrak{N}_{2}(|\varphi|)^{q}(x)dx\right)^{\frac{1}{q}}&\lesssim\left(\int_{\R^{2}}\frac{\varphi^{2}(x)}{(1+|x|^{2})^{4}}dx\right)^{\frac{1}{2}}<+\infty.
    \end{aligned}
\end{equation}
Consequently, $\mathfrak{N}(\varphi) \in L^{q}(\mathbb{R}^{2})$. By the elliptic regularity theory \cite{gilbarg1977elliptic} and the Sobolev embedding theorem, we conclude that $\varphi \in L^{\infty}_{\mathrm{loc}}(\mathbb{R}^{2})$. 

Now, we consider the Kelvin transform of $\varphi$, namely
\begin{equation}
    \psi(x):=\varphi\left(\frac{x}{|x|^{2}}\right).
\end{equation}
Then using \eqref{eq-2} and the identity $|x||\frac{x}{|x|^{2}}-y|=|y||x-\frac{y}{|y|^{2}}|$, we get that
\begin{equation}
\begin{aligned}
     -\Delta\psi(x)&=-\frac{1}{|x|^{4}}\Delta\varphi\left(\frac{x}{|x|^{2}}\right)\\
     &=\frac{C_{\alpha}^{4-\alpha}}{\big(1+|x|^{2}\big)^{\frac{4-\alpha}{2}}}\int_{\R^2}\frac{\varphi(y)}{(|x||\frac{x}{|x|^{2}}-y|)^{\alpha}(1+|y|^{2})^{\frac{4-\alpha}{2}}}dy\\
     &\quad+\frac{C_{\alpha}^{4-\alpha}\varphi(\frac{x}{|x|^{2}})}{\big(1+|x|^{2}\big)^{\frac{4-\alpha}{2}}}\int_{\R^2}\frac{1}{(|x||\frac{x}{|x|^{2}}-y|)^{\alpha}(1+|y|^{2})^{\frac{4-\alpha}{2}}}dy\\
     &=\frac{C_{\alpha}^{4-\alpha}}{\big(1+|x|^{2}\big)^{\frac{4-\alpha}{2}}}\int_{\R^2}\frac{\varphi(y)}{(|y||x-\frac{y}{|y|^{2}}|)^{\alpha}(1+|y|^{2})^{\frac{4-\alpha}{2}}}dy\\
     &\quad+\frac{C_{\alpha}^{4-\alpha}\varphi(\frac{x}{|x|^{2}})}{\big(1+|x|^{2}\big)^{\frac{4-\alpha}{2}}}\int_{\R^2}\frac{1}{(|y||x-\frac{y}{|y|^{2}}|)^{\alpha}(1+|y|^{2})^{\frac{4-\alpha}{2}}}dy\\
     &=\left(\displaystyle{\int_{\R^2}}\frac{e^{U(y)}\psi(y)}{|x-y|^{\alpha}}dy\right)e^{U(x)}+\left(\displaystyle{\int_{\R^2}}\frac{e^{U(y)}}{|x-y|^{\alpha}}dy\right)e^{U(x)}\psi(x).
\end{aligned}    
\end{equation}
Hence $\psi$ satisfies the same equation as $\varphi$ and then $\psi\in L^{\infty}_{loc}(\R^{2})$. This implies that $\varphi\in L^{\infty}(\R^{2})$.
\end{proof}

\begin{Prop}
Assume that $\alpha\in (0,2)$, and $\varphi\in L^{2}_{w}(\mathbb{R}^2)$ is a distributional solution to the equation \eqref{eq-2}. Then $\varphi\in C^{\infty}(\R^{2})$ is a classical solution.
\end{Prop}
\begin{proof}
First, by \eqref{improtant-estimate-1} and Proposition \ref{prop-distributional}, we have
\begin{equation}  \label{eq-decay-estimate}
|\mathfrak{N}(\varphi)(x)| \lesssim \left(\int_{\mathbb{R}^2} \frac{e^{U(y)}}{|x-y|^{\alpha}} \, \mathrm{d}y\right) e^{U(x)} \lesssim \frac{1}{(1+|x|^{2})^{2}}.  
\end{equation}  
Moreover, integrating over $\mathbb{R}^{2}$ gives that
\begin{equation}  
\int_{\mathbb{R}^{2}} |\mathfrak{N}(\varphi)(x)| \, \mathrm{d}x \lesssim \int_{\mathbb{R}^{2}} \frac{1}{(1+|x|^{2})^{2}} \, \mathrm{d}x <+\infty.  
\end{equation}  
This shows that $\mathfrak{N}(\varphi) \in L^{\infty}(\mathbb{R}^{2}) \cap L^{1}(\mathbb{R}^{2})$. Then, Lemma \ref{lem-K-2} implies that the function 
\begin{equation}\label{proof-regularity-2-3}
    P(x):=\varphi(x)-K(\mathfrak{N}(\varphi))(x)
\end{equation}
is a distributional solution to $-\Delta P=0$ in $\R^{2}$. Applying Weyl's Lemma together with Lemma \ref{lem-K-1}, we conclude that $P\in C^{\infty}(\R^{2})$ and $\varphi(x)=K(\mathfrak{N}(\varphi))(x)+P(x)\in C^{1}(\R^{2})$.

Next, we prove by induction that $\varphi\in C^{\infty}(\R^{2})$. Let \(\eta \in C^\infty(\mathbb{R}^2)\) satisfy both \(\eta \equiv 0\) in \(B_1(0)\) and \(\eta \equiv 1\) in \(\mathbb{R}^2 \setminus B_2(0)\) and, for \(i \in \{1, 2\}\) set
$K_i(y):= -\frac{1}{2\pi} \frac{y_i}{|y|^2}\text{~for~} y \in \mathbb{R}^2 \setminus \{0\}$. Then
\begin{equation}\label{proof-regularity-2-4}
    \begin{aligned}
        \frac{\partial}{\partial x_i} K(\mathfrak{N}(\varphi))(x)&= \int_{\mathbb{R}^2} \eta(x-y)K_i(x-y)\mathfrak{N}(\varphi)(y) \, dy\\
        &\quad+ \int_{\mathbb{R}^2} (1-\eta(y))K_i(y)\mathfrak{N}(\varphi)(x-y) \, dy.
    \end{aligned}
\end{equation}
For \(k\in\{1,2,\ldots\}\), assume by induction that \(\varphi\in C^{k}(\mathbb{R}^{2})\) has been established and let \(\beta=(\beta_{1},\beta_{2})\) be a multi-index with \(|\beta|=k\). Let \(e_{i}\in\mathbb{N}^{2}\) denote the usual multi-index of order one corresponding to \(\frac{\partial}{\partial x_{i}}\), then by \eqref{proof-regularity-2-4}, we have
\begin{equation}
    \begin{aligned}
\partial^{\beta+e_{i}}K(\mathfrak{N}(\varphi))(x) 
&= \int_{\mathbb{R}^{2}}\partial_{x}^{\beta}\left(\eta(x-y)K_{i}(x-y)\right)\mathfrak{N}(\varphi)(y)\,\mathrm{d}y \\
&\quad + \int_{\mathbb{R}^{2}}(1-\eta(y))K_{i}(y)\partial_{x}^{\beta}\left(\mathfrak{N}(\varphi)(x-y)\right)\,\mathrm{d}y \\
&:=I_{1}(x) + I_{2}(x).
\end{aligned}
\end{equation}
Since \(\partial^{\beta}(\eta K_{i})\in L^{\infty}(\R^{2})\cap C^{\infty}(\mathbb{R}^{2})\) and \(\mathfrak{N}(\varphi)\in L^{1}(\mathbb{R}^{2})\), the Dominated Convergence Theorem guarantees that \(I_{1}(x)\in C^{0}(\mathbb{R}^{2})\). On the other hand, since $\varphi\in C^{k}(\mathbb{R}^{2})\cap L^{\infty}(\R^{2})$, we observe that
\begin{equation}
    \begin{aligned}
        e^{U(y)}\varphi(y)=\left(\frac{C_{\alpha}}{1+|y|^{2}}\right)^{\frac{4-\alpha}{2}}\varphi(y)\in L^{\frac{4}{4-\alpha}}(\R^{2})\cap L^{\infty}(\R^{2})\cap C^{k}(\R^{2}),
    \end{aligned}
\end{equation}
and \cite[Lemma 2.8]{Gluck2025dcds} implies that $\mathfrak{N}_{1}(\varphi)(x)\in C^{k}(\R^{2})$. Moreover, by \eqref{improtant-estimate-1}
\begin{equation}
    \begin{aligned}
        \mathfrak{N}_{2}(\varphi)(x)&=\left(\displaystyle{\int_{\R^2}}\frac{e^{U(y)}}{|x-y|^{\alpha}}dy\right)e^{U(x)}\varphi(x)=\frac{2(4-\alpha)}{C_{\alpha}^{2}}e^{\frac{4}{4-\alpha}U(x)}\varphi(x)\\
        &=\frac{2(4-\alpha)}{(1+|x|^{2})^{2}}\varphi(x)\in C^{k}(\R^{2}).
    \end{aligned}
\end{equation}
So that $\mathfrak{N}(\varphi)=\mathfrak{N}_{1}(\varphi)+\mathfrak{N}_{2}(\varphi)\in C^{k}(\mathbb{R}^{2})$. Hence, for every \(x\in\mathbb{R}^{2}\), we have
\begin{equation}
    \begin{aligned}
        \Big{|}(1-\eta(y))K_{i}(y)\partial^{\beta}(\mathfrak{N}(\varphi))(x-y)\Big{|}&\leq \frac{\chi_{B_{2}}(y)}{|y|}\|\mathfrak{N}(\varphi)\|_{C^{k}(B_{2}(x))}\in L^{1}(\mathbb{R}^{2};\mathrm{d}y),
    \end{aligned}
\end{equation}
so the Dominated Convergence Theorem guarantees the continuity of \(I_{2}(x)\). Having verified the continuity of \(\partial^{\beta+e_{i}}(K(\mathfrak{N}(\varphi)))\) for arbitrary multi-index $\beta+e_{i}$ of order \(k+1\), we conclude that \(K(\mathfrak{N}(\varphi))\in C^{k+1}(\mathbb{R}^{2})\). Equation \eqref{proof-regularity-2-3} together with the fact that \(P\in C^{\infty}(\mathbb{R}^{2})\) now implies that \(u\in C^{k+1}(\mathbb{R}^{2})\) for any \(k\in\{1,2,\ldots\}\). This completes the proof.
\end{proof}

The integral representation of solutions to \eqref{eq-2} is as follows.

\begin{Prop}\label{prop-integral representation}Assume that $\alpha\in (0,2)$, and $\varphi\in L^{2}_{w}(\mathbb{R}^2)$ is a distributional solution to the equation \eqref{eq-2}. Then
\begin{equation}\label{integral condition}
     \int_{\R^{2}}\mathfrak{N}(\varphi)(y)dy=0,
\end{equation}
and there exists a constant $C_{\varphi}:=\lim_{|x|\to+\infty}\varphi(x)$ such that
\begin{equation}\label{green-frormula}
    \varphi(x)=-\frac{1}{2\pi}\int_{\R^{2}}\log(|x-y|)\mathfrak{N}(\varphi)(y)dy+C_{\varphi}.
\end{equation}
\end{Prop}
\begin{proof}
From \eqref{proof-regularity-2-3}, we have
\begin{equation}
    \varphi(x)=K(\mathfrak{N}(\varphi))(x)+P(x),
\end{equation}
where the operator $K$ is defined in \eqref{defin-K} and $P$ is a harmonic function satisfying $-\Delta P=0$ in $\R^{2}$.  Now, we claim that 
\begin{equation}\label{eq-decay-K-varphi}
    |K(\mathfrak{N}(\varphi))(x)|\lesssim 1+\log(1+|x|)\text{~for any~}x\in\R^{2}.
\end{equation}
Notice that
\begin{equation}
    \begin{aligned}
       K(\mathfrak{N}(\varphi))(x)
       &=\frac{1}{2\pi}\int_{\R^{2}}\log(1+|y|)\mathfrak{N}(\varphi)(y)dy-\frac{1}{2\pi}\int_{\R^{2}}\log(|x-y|)\mathfrak{N}(\varphi)(y)dy\\
       &:=K_{1}(\mathfrak{N}(\varphi))+K_{2}(\mathfrak{N}(\varphi))(x).
    \end{aligned}
\end{equation}
Then we need to estimate $K_{1}(\mathfrak{N}(\varphi))$ and $K_{2}(\mathfrak{N}(\varphi))(x)$ separately. First, by \eqref{eq-decay-estimate}, the term $K_{1}(\mathfrak{N}(\varphi))$ can be estimated by
\begin{equation}
    \begin{aligned}
        |K_{1}(\mathfrak{N}(\varphi))|\lesssim\int_{\R^{2}}\frac{\log(1+|y|)}{(1+|y|^{2})^{2}}dy=2\pi\int_{0}^{+\infty}\frac{r\log(1+r)}{(1+r^{2})^{2}}dr=\frac{\pi^{2}}{4}.
    \end{aligned}
\end{equation}
On the other hand, the term $K_{2}(\mathfrak{N}(\varphi))(x)$ can be bounded by
\begin{equation}
    \begin{aligned}
       |K_{2}(\mathfrak{N}(\varphi))(x)|\lesssim \int_{\R^{2}}\frac{|\log|x-y||}{(1+|y|^{2})^{2}}dy=\int_{\R^{2}}\frac{|\log|y||}{(1+|x-y|^{2})^{2}}dy.
    \end{aligned}
\end{equation}
The term $K_{2}(\mathfrak{N}(\varphi))(x)$ will be estimated by several cases. \\
Case 1: $|x|\leq 2$.
\begin{equation}
    \begin{aligned}
        \int_{\R^{2}}\frac{|\log|y||}{(1+|x-y|^{2})^{2}}dy&=\int_{|y|\leq 4}\frac{|\log|y||}{(1+|x-y|^{2})^{2}}dy+\int_{|y|\geq 4}\frac{\log|y|}{(1+|x-y|^{2})^{2}}dy\\
        &\lesssim \int_{|y|\leq 4}|\log|y||dy+\int_{|y|\geq 4}\frac{\log|y|}{(1+|y|^{2})^{2}}dy\lesssim 1.\\
    \end{aligned}
\end{equation}
Case 2: $|x|\geq 2$ and $|y|\geq 2|x|$. Then $|x-y|\geq |y|-|x|\geq \frac{|y|}{2}$ and
\begin{equation}
    \begin{aligned}
        \int_{|y|\geq 2|x|}\frac{|\log|y||}{(1+|x-y|^{2})^{2}}dy&\lesssim \int_{|y|\geq 4}\frac{\log|y|}{(1+|y|^{2})^{2}}dy\lesssim1.
    \end{aligned}
\end{equation}
Case 3: $|x|\geq 2$ and $|y|\leq \frac{|x|}{2}$. Then $|x-y|\geq |x|-|y|\geq \frac{|x|}{2}$ and
\begin{equation}
    \begin{aligned}
        \int_{|y|\leq\frac{|x|}{2}}\frac{|\log|y||}{(1+|x-y|^{2})^{2}}dy&\lesssim \frac{1}{(1+|x|^{2})^{2}}\int_{|y|\leq 1}|\log|y||dy\\&\quad+\frac{1}{(1+|x|^{2})^{2}}\int_{1\leq|y|\leq\frac{|x|}{2}}\log|y|dy\lesssim1.
    \end{aligned}
\end{equation}
Case 4: $|x|\geq 2$ and $\frac{|x|}{2}\leq|y|\leq 2|x|$. Then $|y|\geq1$ and
\begin{equation}
    \begin{aligned}
        \int_{\frac{|x|}{2}\leq|y|\leq 2|x|}\frac{|\log|y||}{(1+|x-y|^{2})^{2}}dy\lesssim \log(2|x|) \int_{\R^{2}}\frac{1}{(1+|y|^{2})^{2}}dy\lesssim 1+ \log(1+|x|).
    \end{aligned}
\end{equation}
The combination of these estimates establishes \eqref{eq-decay-K-varphi}, yielding
\begin{equation}
    P(x)=\varphi(x)-K(\mathfrak{N}(\varphi))(x)\lesssim 1+\log(1+|x|)\text{~for any~}x\in\R^{2}.
\end{equation}
In conjunction with \cite[Lemma 2.12]{Gluck2025dcds}, we conclude that $P$ is a constant and the solution $\varphi$ admits the integral representation
\begin{equation}
\begin{aligned}
    \varphi(x)&=K_{1}(\mathfrak{N}(\varphi))+K_{2}(\mathfrak{N}(\varphi))(x)+P(x)\\
    &=-\frac{1}{2\pi}\int_{\R^{2}}\log(|x-y|)\mathfrak{N}(\varphi)(y)dy+C_{\varphi}.
\end{aligned} 
\end{equation}
where $C_{\varphi}$ is a constant depending on the solution $\varphi$.

Next, for sufficiently large $|x|$, we have
\begin{equation}\label{green-frormula-proof-1}
    \begin{aligned}
        \varphi(x)
        &=-\frac{1}{2\pi}\log|x|\int_{\R^{2}}\mathfrak{N}(\varphi)(y)dy-\frac{1}{2\pi}\int_{\R^{2}}\log\left(\frac{|x-y|}{|x|}\right)\mathfrak{N}(\varphi)(y)dy+C_{\varphi}.
    \end{aligned}
\end{equation}
Notice that for any given $\delta>0$ small enough
\begin{equation}
    \begin{aligned}
        \int_{\R^{2}}\left|\log\frac{|x-y|}{|x|}\right|\frac{1}{(1+|y|^{2})^{2}}dy&=\int_{\R^{2}}|\log|\frac{x}{|x|}-z||\frac{|x|^{2}}{(1+|x|^{2}|z|^{2})^{2}}dz\\
        &=\int_{|z|\leq\delta}|\log|\frac{x}{|x|}-z||\frac{|x|^{2}}{(1+|x|^{2}|z|^{2})^{2}}dz\\
        &\quad+\int_{|z|\geq \delta}|\log|\frac{x}{|x|}-z||\frac{|x|^{2}}{(1+|x|^{2}|z|^{2})^{2}}dz,
    \end{aligned}
\end{equation}

\begin{equation}
    \begin{aligned}
\int_{|z|\leq\delta}|\log|\frac{x}{|x|}-z||\frac{|x|^{2}}{(1+|x|^{2}|z|^{2})^{2}}dz&\lesssim\int_{|z|\leq\delta}\frac{|z||x|^{2}}{(1+|x|^{2}|z|^{2})^{2}}dz\\
&\leq\frac{1}{|x|}\int_{|y|
\leq |x|\delta}\frac{|y|}{(1+|y|^{2})^{2}}dy\lesssim \frac{1}{|x|}\to0 \text{~as~}|x|\to+\infty,
    \end{aligned}
\end{equation}
and
\begin{equation}
    \begin{aligned}
        \int_{|z|\geq \delta}&|\log|\frac{x}{|x|}-z||\frac{|x|^{2}}{(1+|x|^{2}|z|^{2})^{2}}dz\\
        &\lesssim\frac{1}{|x|^{2}}\int_{|z|\geq \delta}|\log|\frac{x}{|x|}-z||\frac{1}{|z|^{4}}dz\\
        &\lesssim \frac{1}{|x|^{2}}\int_{|\frac{x}{|x|}-z|\leq \delta \cap |z|\geq \delta}|\log|\frac{x}{|x|}-z||\frac{1}{|z|^{4}}dz+\frac{1}{|x|^{2}}\int_{|\frac{x}{|x|}-z|\geq \delta\cap|z|\geq \delta}|\log|\frac{x}{|x|}-z||\frac{1}{|z|^{4}}dz\\
        &\lesssim \frac{1}{|x|^{2}}\int_{|z|\leq \delta}|\log|z||dz+\frac{1}{|x|^{2}}\int_{|z|\geq \delta}\frac{1+\log(1+|z|)}{|z|^{4}}dz\\
        &\lesssim\frac{1}{|x|^{2}}\to0 \text{~as~}|x|\to+\infty.
    \end{aligned}
\end{equation}
Hence,
\begin{equation}\label{green-frormula-proof-2}
    \begin{aligned}
        \lim_{|x|\to+\infty}\left|\int_{\R^{2}}\log\left(\frac{|x-y|}{|x|}\right)\mathfrak{N}(\varphi)(y)dy\right|\leq \lim_{|x|\to+\infty}\int_{\R^{2}}\left|\log\frac{|x-y|}{|x|}\right|\frac{1}{(1+|y|^{2})^{2}}dy=0.
    \end{aligned}
\end{equation}
Since $\varphi$ is bounded, \eqref{green-frormula-proof-1} and \eqref{green-frormula-proof-2} yield that
\begin{equation}
    \int_{\R^{2}}\mathfrak{N}(\varphi)(y)dy=0\text{~and~}C_{\varphi}=\lim_{|x|\to+\infty}\varphi(x).
\end{equation}    
This completes the proof.
\end{proof}



A simple calculation shows that
 
\begin{Lem}\label{prop-1} Let $\varphi_{j}$, $j=1,2,3$ be defined in  \eqref{eq-thm-3}, and let $\rho(x):=\left(\frac{2}{1+|x|^{2}}\right)^{\frac{1}{2}}$. Then we have 
    \begin{equation}
    \begin{aligned}
       \mathcal{S}_{*}\rho(\xi)=(1+\xi_{3})^{\frac{1}{2}}\text{~and~}\mathcal{S}_{*}\varphi_{j}(\xi)=\frac{4-\alpha}{2}\xi_{j},j=1,2,3.
    \end{aligned}
\end{equation}
\end{Lem}
\begin{proof}
By \eqref{defin-S--1}, \eqref{identity-integral} and \eqref{defin-S-*}, we get
\begin{equation}
    \begin{aligned}
        \mathcal{S}_{*}\varphi_{1}(\xi)=\varphi_{1}(\mathcal{S}^{-1}\xi)=\frac{(4-\alpha)\xi_{1}}{1+\xi_{3}}\cdot\frac{1}{1+\frac{\xi_{1}^{2}+\xi^{2}_{2}}{(1+\xi_{3})^{2}}}=\frac{(4-\alpha)\xi_{1}}{1+\xi_{3}}\cdot\frac{1+\xi_{3}}{2}=\frac{4-\alpha}{2}\xi_{1},
    \end{aligned}
\end{equation} 

\begin{equation}
    \begin{aligned}
        \mathcal{S}_{*}\varphi_{2}(\xi)=\varphi_{2}(\mathcal{S}^{-1}\xi)=\frac{(4-\alpha)\xi_{2}}{1+\xi_{3}}\cdot\frac{1}{1+\frac{\xi_{1}^{2}+\xi^{2}_{2}}{(1+\xi_{3})^{2}}}=\frac{(4-\alpha)\xi_{2}}{1+\xi_{3}}\cdot\frac{1+\xi_{3}}{2}=\frac{4-\alpha}{2}\xi_{2},
    \end{aligned}
\end{equation} 
and
\begin{equation}
    \begin{aligned}
        \mathcal{S}_{*}\varphi_{3}(\xi)=\varphi_{3}(\mathcal{S}^{-1}\xi)=\frac{4-\alpha}{2}\cdot\frac{1-\frac{\xi_{1}^{2}+\xi^{2}_{2}}{(1+\xi_{3})^{2}}}{1+\frac{\xi_{1}^{2}+\xi^{2}_{2}}{(1+\xi_{3})^{2}}}=\frac{4-\alpha}{2}\xi_{3}.
    \end{aligned}
\end{equation} 
Similarly, we have
\begin{equation}
    \mathcal{S}_{*}\rho(\xi)=\rho(\mathcal{S}^{-1}\xi)=\left(\frac{2}{1+\frac{\xi_{1}^{2}+\xi^{2}_{2}}{(1+\xi_{3})^{2}}}\right)^{\frac{1}{2}}=(1+\xi_{3})^{\frac{1}{2}}.
\end{equation}
This completes the proof.
\end{proof}

\begin{Rem}\label{remark-1}
Lemma \ref{prop-1} together with \eqref{defin-H-1} implies that the weighted pushforward map $\mathcal{S}_*$ is a one-to-one map 
from the subspace $\operatorname{Span}\{\varphi_j, 1 \leq j \leq 3\} \subset L^\infty(\mathbb{R}^2)$ to the subspace $\mathcal{H}_{1}$, and so is the weighted pullback map $\mathcal{S}^* \colon \mathcal{H}_{1} \to \operatorname{Span}\{\varphi_j, 1 \leq j \leq 3\}$.   
\end{Rem}

Let us denote
\begin{equation}
    \begin{aligned}
       \mathcal{T}_{\mathbb{S}^{2}}\Phi(\xi)= \mathcal{T}_{\mathbb{S}^{2},1}\Phi(\xi)+\mathcal{T}_{\mathbb{S}^{2},2}\Phi(\xi),
    \end{aligned}
\end{equation}
where $\xi\in\mathbb{S}^{2}$, 
\begin{equation}
    \mathcal{T}_{\mathbb{S}^{2},1}\Phi(\xi):=-\frac{1}{2\pi}C_{\alpha}^{4-\alpha}2^{-(4-\alpha)}\int_{\mathbb{S}^{2}}\displaystyle{\int_{\mathbb{S}^{2}}}\log\left(\frac{|\xi-\eta|}{\mathcal{S}_{*}\rho(\xi)\mathcal{S}_{*}\rho(\eta)}\right)\frac{1}{|\eta-\sigma|^{\alpha}}\Phi(\sigma)d\sigma d\eta,
\end{equation}
and
\begin{equation}
    \mathcal{T}_{\mathbb{S}^{2},2}\Phi(\xi):=-\frac{1}{2\pi}C_{\alpha}^{4-\alpha}2^{-(4-\alpha)}\int_{\mathbb{S}^{2}}\displaystyle{\int_{\mathbb{S}^{2}}}\log\left(\frac{|\xi-\eta|}{\mathcal{S}_{*}\rho(\xi)\mathcal{S}_{*}\rho(\eta)}\right)\frac{1}{|\eta-\sigma|^{\alpha}}\Phi(\eta)d\sigma d\eta.
\end{equation}
\begin{Lem}\label{lemma-4.4} Assume that $\alpha\in (0,2)$, and $\varphi\in L^{2}_{w}(\mathbb{R}^2)$ is a distributional solution to the equation \eqref{eq-2}. Then $\mathcal{S}_{*}\varphi \in L^{2}(\mathbb{S}^{2})$, and
\begin{equation}
    \mathcal{S}_{*}\varphi(\xi) = \mathcal{T}_{\mathbb{S}^{2}} \mathcal{S}_{*}\varphi(\xi)+C_{\varphi},
\end{equation}
where $C_{\varphi}:=\lim_{|x|\to+\infty}\varphi(x)$ defined in Proposition \ref{prop-integral representation}.
\end{Lem}
\begin{proof}
By \eqref{identity-integral} and \eqref{defin-S-*}, we have
\begin{equation}
    \begin{aligned}
        \int_{\mathbb{S}^{2}}| \mathcal{S}_{*}\varphi(\xi)|^{2}d\xi=\int_{\R^{2}}|\varphi(x)|^{2}\rho^{4}(x)dx\lesssim\int_{\R^{2}}\frac{1}{(1+|x|^{2})^{2}}dx<+\infty.
    \end{aligned}
\end{equation}
Moreover, by \eqref{green-frormula}, we obtain that
\begin{equation}\label{lemma-4.4-proof-2}
\begin{aligned}
    \varphi(x)&=-\frac{1}{2\pi}\int_{\R^{2}}\log(|x-y|)\mathfrak{N}(\varphi)(y)dy+C_{\varphi}\\
     &=-\frac{1}{2\pi}\int_{\R^{2}}\log(|x-y|)\mathfrak{N}_{1}(\varphi)(y)dy-\frac{1}{2\pi}\int_{\R^{2}}\log(|x-y|)\mathfrak{N}_{2}(\varphi)(y)dy+C_{\varphi}.
\end{aligned}
\end{equation}
Notice that by \eqref{eq-S-x-S-y-x-y} and \eqref{eq-defin-N-varphi}
    \begin{equation}
    \begin{aligned}
        \mathfrak{N}_{1}(\varphi)(y)&=C_{\alpha}^{4-\alpha}\left(\displaystyle{\int_{\R^2}}\frac{\varphi(z)}{|y-z|^{\alpha}(|1+|z|^{2})^{\frac{4-\alpha}{2}}}dz\right)\frac{1}{(1+|y|^{2})^{\frac{4-\alpha}{2}}}\\
        &=C_{\alpha}^{4-\alpha}2^{-(4-\alpha)}\rho^{4}(y)\displaystyle{\int_{\R^2}}\frac{\varphi(z)}{|\mathcal{S}y-\mathcal{S}z|^{\alpha}}\rho^{4}(z)dz
    \end{aligned}
\end{equation}
and
\begin{equation}\label{lemma-4.4-proof-4}
    \begin{aligned}
         \mathfrak{N}_{2}(\varphi)(y)&=C_{\alpha}^{4-\alpha}\left(\displaystyle{\int_{\R^2}}\frac{1}{|y-z|^{\alpha}(|1+|z|^{2})^{\frac{4-\alpha}{2}}}dz\right)\frac{\varphi(y)}{(1+|y|^{2})^{\frac{4-\alpha}{2}}}\\
        &=C_{\alpha}^{4-\alpha}2^{-(4-\alpha)}\rho^{4}(y)\varphi(y)\displaystyle{\int_{\R^2}}\frac{1}{|\mathcal{S}y-\mathcal{S}z|^{\alpha}}\rho^{4}(z)dz.
    \end{aligned}
\end{equation}
Combining \eqref{lemma-4.4-proof-2}--\eqref{lemma-4.4-proof-4}, we derive that
\begin{equation}
    \begin{aligned}
    \varphi(x)&=-\frac{1}{2\pi}C_{\alpha}^{4-\alpha}2^{-(4-\alpha)}\int_{\R^{2}}\displaystyle{\int_{\R^2}}\frac{\log(|x-y|)}{|\mathcal{S}y-\mathcal{S}z|^{\alpha}}\varphi(z)\rho^{4}(y)\rho^{4}(z)dydz\\
    &\quad-\frac{1}{2\pi}C_{\alpha}^{4-\alpha}2^{-(4-\alpha)}\int_{\R^{2}}\displaystyle{\int_{\R^2}}\frac{\log(|x-y|)}{|\mathcal{S}y-\mathcal{S}z|^{\alpha}}\varphi(y)\rho^{4}(y)\rho^{4}(z)dydz+C_{\varphi}\\
    &=-\frac{1}{2\pi}C_{\alpha}^{4-\alpha}2^{-(4-\alpha)}\int_{\mathbb{S}^{2}}\displaystyle{\int_{\mathbb{S}^{2}}}\log\left(\frac{|\mathcal{S}x-\eta|}{\mathcal{S}_{*}\rho(\mathcal{S}x)\mathcal{S}_{*}\rho(\eta)}\right)\frac{1}{|\eta-\sigma|^{\alpha}}\mathcal{S}_{*}\varphi(\sigma)d\sigma d\eta\\
    &\quad-\frac{1}{2\pi}C_{\alpha}^{4-\alpha}2^{-(4-\alpha)}\int_{\mathbb{S}^{2}}\displaystyle{\int_{\mathbb{S}^{2}}}\log\left(\frac{|\mathcal{S}x-\eta|}{\mathcal{S}_{*}\rho(\mathcal{S}x)\mathcal{S}_{*}\rho(\eta)}\right)\frac{1}{|\eta-\sigma|^{\alpha}}\mathcal{S}_{*}\varphi(\eta)d\sigma d\eta+C_{\varphi},
\end{aligned}
\end{equation}
and consequently,
\begin{equation}
    \begin{aligned}
       \mathcal{S}_{*}\varphi(\xi) &=-\frac{1}{2\pi}C_{\alpha}^{4-\alpha}2^{-(4-\alpha)}\int_{\mathbb{S}^{2}}\displaystyle{\int_{\mathbb{S}^{2}}}\log\left(\frac{|\xi-\eta|}{\mathcal{S}_{*}\rho(\xi)\mathcal{S}_{*}\rho(\eta)}\right)\frac{1}{|\eta-\sigma|^{\alpha}}\mathcal{S}_{*}\varphi(\sigma)d\sigma d\eta\\
         &\quad-\frac{1}{2\pi}C_{\alpha}^{4-\alpha}2^{-(4-\alpha)}\int_{\mathbb{S}^{2}}\displaystyle{\int_{\mathbb{S}^{2}}}\log\left(\frac{|\xi-\eta|}{\mathcal{S}_{*}\rho(\xi)\mathcal{S}_{*}\rho(\eta)}\right)\frac{1}{|\eta-\sigma|^{\alpha}}\mathcal{S}_{*}\varphi(\eta)d\sigma d\eta+C_{\varphi}\\
         &=\mathcal{T}_{\mathbb{S}^{2}} \mathcal{S}_{*}\varphi(\xi)+C_{\varphi}.
    \end{aligned}
\end{equation}
This completes the proof.
\end{proof}

We now classify the solutions to \eqref{eq-2} on $\mathbb{S}^{2}$ via the integral equations and the spherical harmonic decomposition.

\begin{Prop} \label{prop-4.5}
Assume that $\alpha\in (0,2)$, and $\varphi\in L^{2}_{w}(\mathbb{R}^2)$ is a distributional solution to the equation \eqref{eq-2}. Then
\begin{equation}
    \mathcal{S}_{*}\varphi\in\mathcal{H}_{1}.
\end{equation}
\end{Prop}
\begin{proof}
Notice that $\mathcal{S}_{*}\varphi \in L^{2}(\mathbb{S}^{2})$, the orthogonal decomposition in \eqref{orthogonal decomposition} yields
\begin{equation}
    \mathcal{S}_{*}\varphi(\xi)=\sum_{k=0}^{\infty}\sum_{j=1}^{2k+1}\Phi_{k,j}Y_{k,j}(\xi),
\end{equation}
where $\Phi_{k,j}:=\int_{\mathbb{S}^{2}} \mathcal{S}_{*}\varphi(\xi)Y_{k,j}(\xi)d\xi$. When $k=0$, by \eqref{improtant-estimate-1}, \eqref{identity-integral}, and \eqref{integral condition}
\begin{equation}
\begin{aligned}
   \Phi_{0,1}&=\frac{1}{2\sqrt{\pi}}\int_{\mathbb{S}^{2}} \mathcal{S}_{*}\varphi(\xi)d\xi=\frac{1}{2\sqrt{\pi}}\int_{\R^{2}}\varphi(x)\rho^{4}(x)dx\\
   &=\frac{2}{\sqrt{\pi}C_{\alpha}^{2}}\int_{\R^{2}}e^{\frac{4}{4-\alpha}U(x)}\varphi(x)dx=\frac{1}{2(4-\alpha)\sqrt{\pi}}\int_{\R^{2}}\mathfrak{N}(\varphi)(x)dx=0. 
\end{aligned}   
\end{equation} 
When $k\geq 1$, by Lemma \ref{lemma-4.4}, we have
\begin{equation}
    \begin{aligned}
        \Phi_{k,j}&=\int_{\mathbb{S}^{2}} \mathcal{S}_{*}\varphi(\xi)Y_{k,j}(\xi)d\xi=\int_{\mathbb{S}^{2}} \mathcal{T}_{\mathbb{S}^{2}}\mathcal{S}_{*}\varphi(\xi)Y_{k,j}(\xi)d\xi+C_{\varphi}\int_{\mathbb{S}^{2}} Y_{k,j}(\xi)d\xi\\
        &=-\frac{1}{2\pi}C_{\alpha}^{4-\alpha}2^{-(4-\alpha)}\int_{\mathbb{S}^{2}}\int_{\mathbb{S}^{2}}\displaystyle{\int_{\mathbb{S}^{2}}}\log\left(\frac{|\xi-\eta|}{\mathcal{S}_{*}\rho(\xi)\mathcal{S}_{*}\rho(\eta)}\right)\frac{1}{|\eta-\sigma|^{\alpha}}\mathcal{S}_{*}\varphi(\sigma) Y_{k,j}(\xi)d\sigma d\eta d\xi\\
        &\quad-\frac{1}{2\pi}C_{\alpha}^{4-\alpha}2^{-(4-\alpha)}\int_{\mathbb{S}^{2}}\int_{\mathbb{S}^{2}}\displaystyle{\int_{\mathbb{S}^{2}}}\log\left(\frac{|\xi-\eta|}{\mathcal{S}_{*}\rho(\xi)\mathcal{S}_{*}\rho(\eta)}\right)\frac{1}{|\eta-\sigma|^{\alpha}}\mathcal{S}_{*}\varphi(\eta) Y_{k,j}(\xi)d\sigma d\eta d\xi.
    \end{aligned}
\end{equation}
Then by Lemma \ref{lem-2.4}, we know that
\begin{equation}
    \begin{aligned}
        &\int_{\mathbb{S}^{2}}\displaystyle{\int_{\mathbb{S}^{2}}}\log\left(\frac{|\xi-\eta|}{\mathcal{S}_{*}\rho(\xi)\mathcal{S}_{*}\rho(\eta)}\right)\frac{1}{|\eta-\sigma|^{\alpha}}Y_{k,j}(\xi)d\eta d\xi\\
        &=\int_{\mathbb{S}^{2}}\displaystyle{\int_{\mathbb{S}^{2}}}\log(|\xi-\eta|)\frac{1}{|\eta-\sigma|^{\alpha}}Y_{k,j}(\xi)d\eta d\xi\\
        &\quad-\int_{\mathbb{S}^{2}}\displaystyle{\int_{\mathbb{S}^{2}}}\log(\mathcal{S}_{*}\rho(\xi))\frac{1}{|\eta-\sigma|^{\alpha}}Y_{k,j}(\xi)d\eta d\xi\\
        &\quad-\int_{\mathbb{S}^{2}}\displaystyle{\int_{\mathbb{S}^{2}}}\log(\mathcal{S}_{*}\rho(\eta))\frac{1}{|\eta-\sigma|^{\alpha}}Y_{k,j}(\xi)d\eta d\xi\\
        &=\tilde{\mu}_{k}\mu_{k}(\alpha)Y_{k,j}(\sigma)-\mu_{0}(\alpha)\displaystyle{\int_{\mathbb{S}^{2}}}\log(\mathcal{S}_{*}\rho(\xi))Y_{k,j}(\xi)d\xi
    \end{aligned}
\end{equation}
and
\begin{equation}
    \begin{aligned}
        &\int_{\mathbb{S}^{2}}\displaystyle{\int_{\mathbb{S}^{2}}}\log\left(\frac{|\xi-\eta|}{\mathcal{S}_{*}\rho(\xi)\mathcal{S}_{*}\rho(\eta)}\right)\frac{1}{|\eta-\sigma|^{\alpha}} Y_{k,j}(\xi)d\sigma d\xi\\
         &=\int_{\mathbb{S}^{2}}\displaystyle{\int_{\mathbb{S}^{2}}}\log(|\xi-\eta|)\frac{1}{|\eta-\sigma|^{\alpha}}Y_{k,j}(\xi)d\sigma d\xi\\
        &\quad-\int_{\mathbb{S}^{2}}\displaystyle{\int_{\mathbb{S}^{2}}}\log(\mathcal{S}_{*}\rho(\xi))\frac{1}{|\eta-\sigma|^{\alpha}}Y_{k,j}(\xi)d\sigma  d\xi\\
        &\quad-\int_{\mathbb{S}^{2}}\displaystyle{\int_{\mathbb{S}^{2}}}\log(\mathcal{S}_{*}\rho(\eta))\frac{1}{|\eta-\sigma|^{\alpha}}Y_{k,j}(\xi)d\sigma  d\xi\\
        &=\tilde{\mu}_{k}\mu_{0}(\alpha)Y_{k,j}(\eta)-\mu_{0}(\alpha)\displaystyle{\int_{\mathbb{S}^{2}}}\log(\mathcal{S}_{*}\rho(\xi))Y_{k,j}(\xi)d\xi.  
    \end{aligned}
\end{equation}
Hence, we have
\begin{equation}
    \begin{aligned}
         \Phi_{k,j}&=-\frac{1}{2\pi}C_{\alpha}^{4-\alpha}2^{-(4-\alpha)}\tilde{\mu}_{k}\mu_{k}
         (\alpha)\int_{\mathbb{S}^{2}}\mathcal{S}_{*}\varphi(\sigma)Y_{k,j}(\sigma)
         d\sigma \\
         &\quad-\frac{1}{2\pi}C_{\alpha}^{4-\alpha}2^{-(4-\alpha)}\tilde{\mu}_{k}\mu_{0}
         (\alpha)\int_{\mathbb{S}^{2}}\mathcal{S}_{*}\varphi(\eta)Y_{k,j}(\eta)
         d\eta \\
         &\quad+\frac{1}{2\pi}C_{\alpha}^{4-\alpha}2^{-(4-\alpha)}\mu_{0}(\alpha)\left(\displaystyle{\int_{\mathbb{S}^{2}}}\log(\mathcal{S}_{*}\rho(\xi))Y_{k,j}(\xi)d\xi\right)\left(\int_{\mathbb{S}^{2}}\mathcal{S}_{*}\varphi(\sigma)
         d\sigma \right)\\
          &\quad+\frac{1}{2\pi}C_{\alpha}^{4-\alpha}2^{-(4-\alpha)}\mu_{0}(\alpha)\left(\displaystyle{\int_{\mathbb{S}^{2}}}\log(\mathcal{S}_{*}\rho(\xi))Y_{k,j}(\xi)d\xi\right)\left(\int_{\mathbb{S}^{2}}\mathcal{S}_{*}\varphi(\eta)
         d\eta\right)\\
         &=-\frac{1}{2\pi}C_{\alpha}^{4-\alpha}2^{-(4-\alpha)}\tilde{\mu}_{k}(\mu_{k}(\alpha)+\mu_{0}(\alpha))\Phi_{k,j}.
    \end{aligned}
\end{equation}
When $k=1$,
\begin{equation}
    \begin{aligned}
        -\frac{1}{2\pi}C_{\alpha}^{4-\alpha}2^{-(4-\alpha)}\tilde{\mu}_{1}(\mu_{1}(\alpha)+\mu_{0}(\alpha))=\frac{1}{2^{5-\alpha}}\frac{(2-\alpha)(4-\alpha)}{\pi}\frac{2^{5-\alpha}\pi}{(2-\alpha)(4-\alpha)}=1.
    \end{aligned}
\end{equation}
When $k\geq 2$,
\begin{equation}
    \begin{aligned}
        -\frac{1}{2\pi}&C_{\alpha}^{4-\alpha}2^{-(4-\alpha)}\tilde{\mu}_{k}(\mu_{k}(\alpha)+\mu_{0}(\alpha))\\
        &=\frac{1}{2^{4-\alpha}k(k+1)}\frac{(2-\alpha)(4-\alpha)}{\pi}\left(2^{2-\alpha} \pi \dfrac{\Gamma\left(k+ \frac{\alpha}{2} \right) \Gamma\left(\frac{2-\alpha}{2} \right)}{\Gamma\left( \frac{\alpha}{2} \right) \Gamma\left( k +2- \frac{ \alpha}{2} \right)}+\frac{2^{3-\alpha}\pi}{2-\alpha}\right)\\
        &<\frac{1}{2^{4-\alpha}k(k+1)}\frac{(2-\alpha)(4-\alpha)}{\pi}\frac{2^{5-\alpha}\pi}{(2-\alpha)(4-\alpha)}\\
        &=\frac{2}{k(k+1)}<1.
    \end{aligned}
\end{equation}
Hence $\Phi_{k,j}=0$ for any $k\not=1$ and so $\mathcal{S}_{*}\varphi\in\mathcal{H}_{1}$. This completes the proof.
\end{proof}

\noindent{\bf{Proof of Theorem \ref{thm-1}}.}
By Proposition \ref{prop-4.5} and \eqref{defin-H-1}, we have
    \begin{equation}
        \mathcal{S}_{*}\varphi\in\text{span}\{ \xi_j \mid 1 \leqslant j \leqslant 3 \}.
    \end{equation}
    Moreover, by Lemma \ref{prop-1} and Remark \ref{remark-1}, we obtain
    \begin{equation}
        \varphi\in\text{span}\{ \varphi_j \mid 1 \leqslant j \leqslant 3 \}.
    \end{equation}
Next, it suffices to show that equation \eqref{eq-thm-4} holds. First, by \eqref{eq-thm-3}, we have $|\nabla\varphi(x)|\lesssim \frac{1}{|x|^{2}}$ for any $|x|$ large enough. Then by \eqref{improtant-estimate-1}, \eqref{eq-2} and the Gauss-Green formula, we have
  \begin{equation}\label{proof-thm-1-1}
    \begin{aligned}
        \int_{\R^{2}}|\nabla \varphi(x)|^{2}dx&=\int_{\R^2}\int_{\R^2}\frac{e^{U(y)}\varphi(y)e^{U(x)}\varphi(x)}{|x-y|^{\alpha}}dydx+\int_{\R^2}\int_{\R^2}\frac{e^{U(y)}e^{U(x)}\varphi^{2}(x)}{|x-y|^{\alpha}}dydx\\
        &=\int_{\R^2}\int_{\R^2}\frac{e^{U(y)}\varphi(y)e^{U(x)}\varphi(x)}{|x-y|^{\alpha}}dydx+\frac{2(4-\alpha)}{C_{\alpha}^{2}}\int_{\R^2}e^{\frac{4}{4-\alpha}U(x)}\varphi^{2}(x)dx.
    \end{aligned}
\end{equation}
Moreover, by the HLS inequality and the H\"older inequality, we obtain
\begin{equation}\label{proof-thm-1-2}
    \begin{aligned}
        \int_{\R^2}\int_{\R^2}\frac{e^{U(y)}\varphi(y)e^{U(x)}\varphi(x)}{|x-y|^{\alpha}}dydx&\leq  \frac{2\pi^{\frac{\alpha}{2}}}{2-\alpha}\left(\int_{\R^2}e^{\frac{4}{4-\alpha}U(x)}|\varphi|^{\frac{4}{4-\alpha}}(x)dx\right)^{\frac{4-\alpha}{2}}\\
        &\leq \frac{2\pi^{\frac{\alpha}{2}}}{2-\alpha}\left(\int_{\R^2}e^{\frac{4}{4-\alpha}U(x)}dx\right)^{\frac{2-\alpha}{2}}\int_{\R^2}e^{\frac{4}{4-\alpha}U(x)}\varphi^{2}(x)dx.
    \end{aligned}
\end{equation}
Combining \eqref{improtant-estimate-2}, \eqref{proof-thm-1-1} and \eqref{proof-thm-1-2}, we have
\begin{equation}
    \begin{aligned}
        \int_{\R^{2}}|\nabla \varphi(x)|^{2}dx&\leq \frac{2\pi^{\frac{\alpha}{2}}}{2-\alpha}\left(\int_{\R^2}e^{\frac{4}{4-\alpha}U(x)}dx\right)^{\frac{2-\alpha}{2}}\int_{\R^2}e^{\frac{4}{4-\alpha}U(x)}\varphi^{2}(x)dx\\
        &\quad+\frac{2(4-\alpha)}{C_{\alpha}^{2}}\int_{\R^2}e^{\frac{4}{4-\alpha}U(x)}\varphi^{2}(x)dx\\
        &=\left(\frac{2\pi}{2-\alpha} C_{\alpha}^{4-\alpha}+2(4-\alpha)\right)\int_{\R^2}\frac{1}{(1+|x|^{2})^{2}}\varphi^{2}(x)dx\\
        &=4(4-\alpha)\int_{\R^2}\frac{1}{(1+|x|^{2})^{2}}\varphi^{2}(x)dx.
    \end{aligned}
\end{equation}
Thus
\begin{equation}
    \begin{aligned}
        \|\varphi\|_{H^{1}_{w}(\R^{2})}^{2}=\|\nabla \varphi\|_{L^{2}(\R^{2})}^{2}+\|\varphi\|_{L^{2}_{w}(\R^{2})}^{2}\leq (4(4-\alpha)+1)\|\varphi\|_{L^{2}_{w}(\R^{2})}^{2},
    \end{aligned}
\end{equation}
which completes the proof.

\appendix
\section{}\label{appendix}
\begin{Prop}
    It holds that $\dot{H}^{1}(\mathbb{R}^{2})\subset L^{2}_{w}(\R^{2})$.
\end{Prop}
\begin{proof} For any $f\in \dot{H}^{1}(\mathbb{R}^{2})$, Theorem~1.48 in \cite{Bahouri2011} implies that $f \in \mathrm{BMO}(\mathbb{R}^{2})$ and there exists a constant $C > 0$ such that
\begin{equation}  
\|f\|_{\mathrm{BMO}(\mathbb{R}^{2})} \leq C \|f\|_{\dot{H}^{1}(\mathbb{R}^{2})}.
\end{equation}  
Moreover, by the John-Nirenberg inequality \cite[Chapter IV]{Stein1993}, we have $f \in L^{p}_{\mathrm{loc}}(\mathbb{R}^{2})$ for any $1 \leq p < \infty$ and there exists a constant $B_{p}>0$ such that
\begin{equation}\label{regularity-proof-2}  
\sup_{Q } \left( \frac{1}{|Q|} \int_{Q} |f(x) - f_{Q}|^{p} \, dx \right)^{1/p} \leq B_{p} \|f\|_{\mathrm{BMO}(\mathbb{R}^{2})},
\end{equation}
where the supremum is taken over all cubes $Q$ in $\mathbb{R}^{2}$. 

Next, let $Q_{k}=Q(0,2^{k})$, $S_{k}=Q_{k}\setminus Q_{k-1}$ for $k\in\mathbb{N}$, $S_{0}=Q_{0}$, and
\begin{equation}
I_{k}=\left(\int_{S_{k}}\frac{|f(x)-f_{Q_{0}}|^{2}}{(1+|x|^{2})^{2}}\,dx\right)^{\frac{1}{2}},\quad k\in\mathbb{N}_{0}=\{0,1,2,\cdots\}.
\end{equation}
Then, we have
\begin{equation}\label{regularity-proof-4}
\left(\int_{\R^{2}}\frac{|f(x)-f_{Q_{0}}|^{2}}{(1+|x|^{2})^{2}}\,dx\right)^{\frac{1}{2}}=I_{0}+\sum_{k=1}^{+\infty}I_{k}.
\end{equation}
By \eqref{regularity-proof-2}, the first term can be controlled as follows
\begin{equation}\label{regularity-proof-5}
I_{0}=\left(\int_{Q_{0}}\frac{|f(x)-f_{Q_{0}}|^{2}}{(1+|x|^{2})^{2}}\,dx\right)^{\frac{1}{2}}\leqslant\left(\int_{Q_{0}}|f(x)-f_{Q_{0}} |^{2}\,dx\right)^{\frac{1}{2}}\leqslant B_{2}|Q_{0}|^{\frac{1}{2}}\,\|f\|_{\textrm{BMO}(\R^{2})}.
\end{equation}
Notice that for $x\in S_{k}$, we have $|x|>2^{k-2}$ and then $(1+|x|^{2})^{2}>2^{4(k-2)}$.
Hence,
\begin{equation}
    \begin{aligned}
        I_{k} &\leqslant 2^{-2(k-2)}\left(\int_{Q_{k}}|f(x)-f_{Q_{0}}|^{2}\,dx\right)^{\frac{1}{2}} \\
&\leqslant  2^{-2(k-2)}\left\{\left(\int_{Q_{k}}|f(x)-f_{Q_{k}}|^{2}\,dx\right)^{\frac{1}{2}}+\left(\int_{Q_{k}}|f_{Q_{k}}-f_{Q_{0}}|^{2}\,dx\right)^{\frac{1}{2}} \right\}\\
&\leqslant2^{-2(k-2)}|Q_{k}|^{\frac{1}{2}}\left(B_{2}\|f\|_{\textrm{BMO}(\R^{2})}+|f_{Q_{k}} -f_{Q_{0}}|\right) \\
&= 2^{4-k}\left(B_{2}\|f\|_{\textrm{BMO}(\R^{2})}+|f_{Q_{k}}-f_{Q_{0}}|\right).
    \end{aligned}
\end{equation}
Moreover, we have
\begin{equation}
    \begin{aligned}
        |f_{Q_{k}}-f_{Q_{0}}| &\leqslant \sum_{i=1}^{k}|f_{Q_{i}}-f_{Q_{i-1}} | \\
&\leqslant \sum_{i=1}^{k}\frac{1}{|Q_{i-1}|}\int_{Q_{i-1}}|f(x)-f_{Q_{i}}|\,dx \\
&\leqslant \sum_{i=1}^{k}\frac{4}{|Q_{i}|}\int_{Q_{i}}|f(x)-f_{Q_{i}}|\,dx \\
&\leqslant 4k\|f\|_{\textrm{BMO}(\R^{2})}.
    \end{aligned}
\end{equation}
Therefore, the second term can be controlled as follows
\begin{equation}\label{regularity-proof-8}
    \sum_{k=1}^{+\infty}I_{k}\leq \|f\|_{\textrm{BMO}(\R^{2})}\sum_{i=1}^{+\infty}(2^{4-k}(B_{2}+4k))\lesssim \|f\|_{\textrm{BMO}(\R^{2})}.
\end{equation}
Combining \eqref{regularity-proof-4}, \eqref{regularity-proof-5} and \eqref{regularity-proof-8}, we obtain that
\begin{equation}
\begin{aligned}
     \left(\int_{\R^{2}}\frac{|f(x)|^{2}}{(1+|x|^{2})^{2}}dx\right)^{\frac{1}{2}}&\leq \left(\int_{\R^{2}}\frac{|f(x)-f_{Q_{0}}|^{2}}{(1+|x|^{2})^{2}}dx\right)^{\frac{1}{2}}+ \left(\int_{\R^{2}}\frac{|f_{Q_{0}}|^{2}}{(1+|x|^{2})^{2}}dx\right)^{\frac{1}{2}}\\
     &\lesssim \|f\|_{\textrm{BMO}(\R^{2})}+|f_{Q_{0}}|\lesssim\|f\|_{\dot{H}^{1}(\R^{2})}+|f_{Q_{0}}|.
\end{aligned}    
\end{equation}
This completes the proof.
\end{proof}


\end{document}